\documentclass[11pt]{amsart}
\usepackage{amscd,amssymb,amsmath}
\usepackage[all]{xy}

\newtheorem{theorem}{Theorem}[section]

\theoremstyle{definition}

\newtheorem{proposition}[theorem]{Proposition}
\newtheorem{corollary}[theorem]{Corollary}

\theoremstyle{remark}
\newtheorem{remark}{Remark}[section]

\numberwithin{equation}{section}
\numberwithin{table}{section}
\begin{document}

\title[Coincidence Wecken property for nilmanifolds]{Coincidence Wecken property for nilmanifolds}

\author{Daciberg Gon\c calves}
\address{Departamento de Matem\'atica - IME - Universidade de S\~ao Paulo,\\
Rua do Mat\~ao, 1010 \\
 CEP 05508-090 -
S\~ao Paulo - SP - Brazil}
\email{dlgoncal@ime.usp.br}

\author{Peter Wong}
\address{Department of Mathematics, Bates College, Lewiston,
ME 04240, U.S.A.}
\email{pwong@bates.edu}

\thanks{The first author was supported in part by  Projeto Tematico  Topologia Algebrica Geometrica e Differencial  2008/57607-6.}

\begin{abstract}
Let $f,g:X\to Y$ be maps from a compact infra-nilmanifold $X$ to a compact nilmanifold $Y$ with $\dim X\ge \dim Y$. In this note, we show that a certain Wecken type property holds, i.e., if the Nielsen number $N(f,g)$ vanishes then $f$ and $g$ are deformable to be coincidence free. We also show that if $X$ is a connected finite complex $X$ and the Reidemeister coincidence number $R(f,g)=\infty$ then $f\sim f'$ so that $C(f',g)=\{x\in X \mid f'(x)=g(x)\}$ is empty.
\end{abstract}

\date{\today}

\maketitle

\newcommand{\af}{\alpha}
\newcommand{\et}{\eta}
\newcommand{\ga}{\gamma}
\newcommand{\ta}{\tau}
\newcommand{\ph}{\varphi}
\newcommand{\bt}{\beta}
\newcommand{\lb}{\lambda}
\newcommand{\wh}{\widehat}
\newcommand{\sg}{\sigma}
\newcommand{\om}{\omega}
\newcommand{\cH}{\mathcal H}
\newcommand{\cF}{\mathcal F}
\newcommand{\N}{\mathcal N}
\newcommand{\R}{\mathcal R}
\newcommand{\Ga}{\Gamma}
\newcommand{\cc}{\mathcal C}
\newcommand{\bea} {\begin{eqnarray*}}
\newcommand{\beq} {\begin{equation}}
\newcommand{\bey} {\begin{eqnarray}}
\newcommand{\eea} {\end{eqnarray*}}
\newcommand{\eeq} {\end{equation}}
\newcommand{\eey} {\end{eqnarray}}
\newcommand{\ovl}{\overline}
\newcommand{\vv}{\vspace{4mm}}
\newcommand{\lra}{\longrightarrow}


\bibliographystyle{amsplain}

\section{Introduction}

Let $f,g:X\to Y$ be maps between topological spaces. The study of the coincidence set $C(f,g)=\{x\in X\mid f(x)=g(x)\}$ has long been a classical problem. Indeed in the 1920s, S. Lefschetz extended his celebrated fixed point theorem to coincidences. More precisely, he considered maps between closed connected orientable manifolds of the same dimension. A homological trace $L(f,g)$ is defined so that $L(f,g)\ne 0$ implies that $C(f,g)\ne \emptyset$. The converse of this result does not hold in general. In the fixed point case, F. Wecken showed that the Nielsen number $N(f)=0$ implies that $f$ is deformable to be fixed point free when $X=Y$ is a compact manifold of dimension $\dim X\ge 3$. In other words, the Nielsen number $N(f)$ is a complete invariant for deforming selfmaps to be fixed point free.

In the same setting as that of Lefschetz, H. Schirmer successfully in 1955 generalized the Nielsen number to coincidences. Moreover, when $\dim X=\dim Y\ge 3$, she showed that $N(f,g)=0 \Rightarrow$ $f\sim f', g\sim g'$ such that $C(f',g')=\emptyset$. A more difficult situation is when $\dim X\ge \dim Y$. In recent years, progress has been made in this positive codimensional coincidence problem. While there are new approaches and new Nielsen type invariants, we do not know, to the best of our knowledge, if our results regarding positive codimensional coincidences (e.g. \cite{GW1,GW2}) have been obtained using other approaches.

When $g=\overline c$ is the constant map at $c\in Y$, the coincidence problem is equivalent to the root problem (since $C(f,g)=f^{-1}(c)$). A very general Wecken type theorem has been established in \cite{GW1}, namely, $N(f;c)=0 \Rightarrow f\sim f'$ such that $f'^{-1}(c)=\emptyset$. Here $N(f;c)$ is the {\it geometric} Nielsen number studied by R. Brooks \cite{B2}. In \cite{GW2}, we showed that if $f,g:N_1\to N_2$ are maps between two compact nilmanifolds with $\dim N_1\ge \dim N_2$, then $N(f,g)>0 \Rightarrow N(f,g)=R(f,g)$ where $R(f,g)$ is the coincidence Reidemeister number (the number of equivalence classes under $u \sim g_{\sharp} (x)uf_{\sharp}(x)^{-1}$ for $x\in \pi_1(N_1), u\in \pi_1(N_2)$) defined at the fundamental group level. Furthermore, if $N(f,g)>0$ then $o_n(f,g)\ne 0$ where $n=\dim N_2$ and $o_n(f,g)$ is the primary obstruction to deforming $f$ and $g$ to be coincidence free on the $n$-th skeleton of $N_1$. See \cite{GJW} for a more detailed treatment of the use of primary obstruction in the coincidence problem.

The purpose of this note is to show that $o_n(f,g)$ is the {\it only} obstruction (that is, there are no higher obstructions), i.e., $o_n(f,g)=0 \Rightarrow$ $f$ and $g$ are deformable to be coincidence free.  Equivalently, for maps between compact nilmanifolds, the Nielsen number $N(f,g)$ is a complete invariant. In fact, we are able to show that if $X$ is a path connected topological space with finitely generated fundamental group and $Y$ is a compact nilmanifold then for any two maps $f,g:X\to Y$, $R(f,g)=\infty \Rightarrow f\sim f', g\sim g'$ with $C(f',g')=\emptyset$. Finally, using a result of Brooks \cite{B1}, one of the two homotopies can be made constant.

\section{Nielsen Coincidence Theory}

For the coincidence problem where the codimension may be positive, we will make use of the more general ({\it geometric}) notion of the Nielsen coincidence number due to Brooks \cite{B2}. Given two maps $f,g: M\to N$, two coincidences $x_1,x_2\in C(f,g)$ are
{\it equivalent} if there is a path $ \alpha:[0,1]\to M$ with
$\alpha(0)=x_1, \alpha(1)=x_2$ such that $f\circ \alpha$ is homotopic 
to $g\circ \alpha$ ($f\circ \alpha \sim g\circ \alpha$) relative to the
endpoints. The equivalence classes are called coincidence classes. If $\{f_{t}\},\{g_{t}\}$ are homotopies of $f,g$ respectively, then $x\in C(f,g)$ and $y\in C(f_{1},g_{1})$ 
are $\{f_{t}\},\{g_{t}\}$-{\it related} if there exists a path 
$C:[0,1]\to M$ with $C(0)=x,C(1)=y$ such that 
$[\{f_{t}\circ C(t)\}]=[\{g_{t}\circ C(t)\}]$ where $[f\circ C]$ denotes 
the fixed-endpoint homotopy class of the image of $C$ under $f$. It 
follows that if a coincidence in a class $\gamma$ of $f,g$ is 
$\{f_{t}\},\{g_{t}\}$-related to a coincidence in a coincidence 
class $\beta$ of $f_{1},g_{1}$, then every coincidence in $\gamma$ is $\{f_{t}\},
\{g_{t}\}$-related to every coincidence in $\beta$. In this case, the 
class $\gamma$ is said to be $\{f_{t}\},\{g_{t}\}$-related to the class 
$\beta$. A coincidence class $\gamma$ of $f,g$ is {\it essential} if for any 
homotopies $\{f_{t}\}$ and $\{g_{t}\}$ of $f$ and $g$, there is a 
coincidence class of $f_{1},g_{1}$ to which it is related.
Finally, the Nielsen coincidence number of $f$ and $g$, denoted by $N(f,g)$, is defined to be the number of essential coincidence classes. This Nielsen number reduces to the classical one when the domain and the target have the same dimension.

As an easy consequence of the main theorem of \cite{GW1}, we obtain the following

\begin{theorem}\label{roots}
Let $f,g:X \to M$ be maps from a compact topological space $X$ to a compact connected topological group $M$. If $N(f,g)=0$ then $f$ and $g$ are deformable to be coincidence free.
\end{theorem}
\begin{proof} It is straightforward to show that the coincidence problem is equivalent to the {\it root} problem concerning the preimage $\varphi^{-1}(e)$ where $e\in M$ is the group unity element and $\varphi(x)=[f(x)]^{-1}\cdot g(x)$ where $\cdot$ denotes the group multiplication in $M$. Then $N(f,g)=0$ iff $N(\varphi; \bar e)=0$. The latter implies that $\varphi$ can be made root free by the main theorem of \cite{GW1}.
\end{proof}

\begin{remark} Theorem \ref{roots} also holds if $M=S^7$, the $7$-sphere, which is not a topological group. The validity of the result is due to the fact that $S^7$ has a multiplication with unique inverse so that the coincidence problem can be transformed into a root problem.
\end{remark}

In extending the remain result of \cite{GW2}, we now prove the main result of this paper.

\begin{theorem}\label{nil-nil}
Let $f,g:N_1\to N_2$ be maps between two compact nilmanifolds with $\dim N_1\ge \dim N_2$. Then $N(f,g)=0$ if and only if $g\sim g'$ with $C(f,g')=\emptyset$.
\end{theorem}
\begin{proof}  By the definition of $N(f,g)$, $C(f,g_1)\ne \emptyset$ for any $g_1\sim g$ if $N(f,g)>0$.

For the converse, suppose $N(f,g)=0$. It follows from \cite[Theorem 4.2]{GW2} that $R(f,g)=\infty$. Since $N_2$ is a nilmanifold, there is a principal $S^1$-bundle $S^1 \hookrightarrow N_2 \stackrel{p}  \to \bar N_2$.
Without loss of generality, we may assume that
$C(p\circ f, p\circ g)$ is a submanifold of dimension $>\dim N_1-\dim N_2$. Note that $C (f,g)\subset C(p\circ f, p\circ g)$. Denote by $\hat f$, $\hat g$ the restrictions of $f,g $ to $C(p\circ f, p\circ g)$.
Since the circle $S^1$ is a group, we have $\hat g(x)=d(x)\cdot \hat f(x)$ for $d: C(p\circ f, p\circ g)\to S^1$ and $x\in C(p\circ f, p\circ g)$. Here we call $d$ the {\it deviation} map. We have two possibilities: $d$ is null homotopic or $d$ is NOT null homotopic.

If $d$ is null homotopic then it is not difficult to show that the
pair  $(f,g)$ can be deformed to be coincidence free, as we will show at the end of the proof. So let us suppose that $d$ is not null homotopic. Next, we will show that under this assumption, $R(f,g)$ must be finite. If $N_2$ is 2-dimensional then it is a torus and the theorem holds. We induct on $\dim N_2$. If $R(p\circ f, p\circ g)$ is infinite then by induction we are done. Suppose $R(\bar f,\bar g)$ is finite where $\bar f=p\circ f, \bar g=p\circ g$.

Next, we show that the number of Reidemeister classes of $f,g$ which project to the Reidemeister class  $[\bar 1]$ (where $\bar 1\in \pi_1(\bar N_2))$ is finite. Let $x\in \pi_1(N_2)$ such that $p_{\sharp}(x)\in [\bar 1]$. Then there exists $u\in \pi_1(N_1)$
such that $p_{\sharp}(x)=\bar g_{\sharp}(u)\bar 1\bar f_{\sharp}(u)^{-1}$ which implies that $p_{\sharp}(g_{\sharp}(u)^{-1}xf_{\sharp}(u))=1$, i.e., the element   $g_{\sharp}(u)^{-1}xf_{\sharp}(u)$ is in the same Reidemeister class of $x$ and it belongs to the fundamental group of the fibre. For $v\in \pi_1(S^1)$, the elements in the Reidemeister class of $v$ are of the form  $g_{\sharp}(z)vf_{\sharp}(z)^{-1}=vg_{\sharp}(z)f_{\sharp}(z)^{-1}$ since $\pi_1(S^1)$ is central in $\pi_1(N_2)$.
For a suitable $z$ this element is an integer $k$ due to the fact that the deviation map $d$ is not null homotopic. In fact $k$ is a generator of the image of $d_{\sharp}$. 
Thus the number of distinct Reidemeister classes is at most finite, which is a contradiction, so   $d$ is null homotopic.

Now we will show that once $d$ is null homotopic, then the pair  $(f,g)$ can be deformed to be coincidence free. 
We will show that $g$ can be deformed to a map $g'$ such that the pair $(f,g')$ is coincidence free.
First observe that $\hat g=g|_{C(p \circ f, p \circ  g)}$  can be deformed to a map $g_1$ such that  $C( \hat f,  g_1)=\emptyset$. This follows because $\hat g=d\cdot \hat f$ and $d$ is homotopic to the constant map at a point $\epsilon \in S^1\setminus \{1\}$.   Consider the constant self homotopy of $p\circ g$ which is the constant homotopy. This homotopy when restricted to $N_1\times \{0\}\cup  C(p \circ f, p \circ  g)\times I$
admits a lift which is the map $g$ in   $N_1\times \{0\}$ and the homotopy between $\hat g$ and $g_1$
in $C(p \circ f, p \circ  g)\times I$ as explained above. By the lifting property of a fibration there is a
lifting  $\tilde g$ of  the constant homotopy which extends the partial lifting. The restriction of the lifting to $N_1\times \{1\}$ gives the desire map $g'$.
\end{proof}

From the proof of Theorem \ref{nil-nil}, we summarize in the following theorem that certain conditions are equivalent.

\begin{theorem}\label{equivalent-conditions}
For any two maps $f,g: N_1\to N_2$ between two compact nilmanifolds with $\dim N_1\ge \dim N_2$, the following are equivalent.
\begin{enumerate}
\item $N(f,g)=0$; \\
\item $R(f,g)=\infty$; \\
\item $f$ and $g$ are deformable to be coincidence free.
\end{enumerate}
\end{theorem}

In the case where $\dim N_1 < \dim N_2$, it is well known that for any pair of maps 
$(f,g)$,  $f$ and $g$ are deformable to be coincidence free. This follows easily using obstruction since 
the homotopy groups $\pi_i(N_2\times N_2, N_2\times N_2\backslash \Delta)=0$  for $i\leq \dim (N_1)$. 
Also for some integer $i>0$, we must have that $\Gamma_i(\pi_1(N_1))/\Gamma_{i+1}(\pi_1(N_1))$ has rank 
less than $\Gamma_i(\pi_1(N_2))/\Gamma_{i+1}(\pi_1(N_2))$  since $\dim N_1 < \dim N_2$. Here, $\Gamma_i(G)$ denotes the $i$-th term in the lower central series of a group $G$. This implies 
that the  number of Reidemeister classes  of the induced homomorphisms on the quotient is infinite so $R(f,g)=\infty$ as well. Thus we conclude that for $\dim N_1 < \dim N_2$ with $N_1, N_2$ nilmanifolds, not only are the three conditions in Theorem \ref{equivalent-conditions} equivelant but also that they always hold.

\begin{corollary}\label{obstruction}
Let $f,g:M\to N$ be two maps between two compact nilmanifolds with $\dim M\ge n=\dim N$. Then $f$ and $g$ are deformable to be coincidence free iff the (primary) obstruction $o_n(f,g)=0$.
\end{corollary}
\begin{proof} The result follows from Theorem \ref{nil-nil} and \cite[Theorem 4.2]{GW2}.
\end{proof}

We now further extend Theorem \ref{nil-nil} where the domain is arbitrary.

\begin{theorem}\label{main}
Let $X$ be a connected topological space with finitely generated $\pi_1(X)$ and $N$ a compact nilmanifold. For any maps $f,g:X\to N$, if $R(f,g)=\infty$ then $f$ and $g$ are deformable to be coincidence free.
\end{theorem}
\begin{proof}
It follows from the proof of \cite[Theorem 3]{GW3} that there is a homotopy commutative diagram
\begin{equation}\label{factor}
\begin{CD}
    X    @.                @>{f,g}>>    @.  N  \\
    @. {\searrow}^q         @.          ^{\bar f, \bar g}{\nearrow} @.     
\end{CD}
\end{equation}
\vskip -6pt
$$\bar N$$
where $\bar N$ is a compact nilmanifold, $R(f,g)=R(\bar f, \bar g)$ and $q_{\#}:\pi_1(X) \to
\pi_1(\bar N)$ is surjective. Since $R(f,g)=\infty$, it follows that $R(\bar f, \bar g)=\infty$. 
If $\dim(\bar N) <\dim(N)$ then certainly $\bar f$, $\bar g$ are deformable  to be coincidence free.
Otherwise, by Theorem  \ref{equivalent-conditions}, $N(\bar f, \bar g) = 0$ and $\bar f, \bar g $ are deformable to be coincidence free.
Note that the coincidences of $f,g$ project to coincidences of $\bar f, \bar g$ under $q$. We conclude that $f,g$ are also deformable to be coincidence free.
\end{proof}

\begin{remark}
The converse of the above 
theorem does not hold in that $f$ and $g$ are deformable to be coincidence free while $R(f,g)$ is finite.
For example, let $S_2$ be the orientable surface of genus 2, $S^n$ the $n$-sphere with $n>2$
and $T^4$ the 4-dimensional torus.  Consider the map $\theta: S_2\to T^4$ such that $\theta_{\#}: \pi_1(S_2)\to \pi_1(T^4)$ is the abelianization (thus a surjective induced homomorphism) and  $ c: S_2\to T^4$ the constant map. 
Let $p: S_2\times S^n \to S_2$ denote  the projection on the first coordinate, $f=\theta \circ p$ and 
$g=c \circ p$ .    Since $f$ and $g$ factor through  $S_2$, which has dimension 2, it follows that $\theta$ and $c$ are deformable to be coincidence free and so are $f$ and $g$. It is straightforward to show that $R(f,g)=1$.
\end{remark}

\begin{corollary}\label{infra}
Let $f,g: M\to N$ be maps where $M$ is a compact infra-nilmanifold and $N$ a compact nilmanifold. Then $N(f,g)=0$ iff $R(f,g)=\infty$. Moreover, if $N(f,g)>0$ then $N(f,g)=R(f,g)$.
\end{corollary}
\begin{proof} Let $p: \tilde M \to M$ be a finite covering where $\tilde M$ is a compact nilmanifold. Furthermore, we may assume without loss of generality that this covering is regular. Define $\tilde f, \tilde g:\tilde M \to N$ by $\tilde f=f\circ p, \tilde g=g\circ p$. Then $N(f,g)=0 \Rightarrow N(\tilde f,\tilde g)=0$. To see this, first note that $p(C(\tilde f, \tilde g))\subseteq C(f,g)$. Let $\tilde x \in C(\tilde f, \tilde g)$. Now let $\gamma$ be the coincidence class of $f$ and $g$ containing $p(\tilde x)$. Since $N(f,g)=0$ there exist $\{f_t\}, \{g_t\}$ such that $p(\tilde x)$ is not related to any coincidences of $f_1$ and $g_1$. It is easy to see that $\tilde x$ is not related to any coincidences of $f_1\circ p$ and $g_1\circ p$. It follows that $R(\tilde f, \tilde g)=\infty$ since $\tilde M$ is a compact nilmanifold. Since $\pi_1(\tilde M)$ is a finite index (normal) subgroup of $\pi_1(M)$, we have the following commutative diagram.
\begin{equation*}\label{exact}
\begin{CD}
    1 @>>> \pi_1(\tilde M)    @>{p_{\sharp}}>>  \pi_1(M) @>>>    F @>>> 1 \\
    @.     @V{\tilde f_{\sharp}}V{\tilde g_{\sharp}}V  @V{f_{\sharp}}V{g_{\sharp}}V   @VVV @.\\
    1 @>>> \pi_1(N)    @>{=}>>  \pi_1(N) @>>>    1 @>>> 1
 \end{CD}
\end{equation*}

Here, $F$ denotes the finite quotient $\pi_1(M)/\pi_1(\tilde M)$. It follows from \cite[Theorem 2.1]{GW4} that $R(f,g)=\infty=R(\tilde f, \tilde g)$. Then the equivalence $N(f,g)=0 \Leftrightarrow R(f,g)=\infty$ follows from Theorem \ref{nil-nil}. Now, suppose $N(f,g)>0$. Then by Theorem  \ref{main}, $R(f,g) < \infty$.
It follows from the commutative diagram above (also \cite[Theorem 2.1]{GW4}) that $R(\tilde f, \tilde g)<\infty$ and hence every coincidence class of $\tilde f$ and $\tilde g$ must be essential. If $x\in C(f,g)$ then for any $\tilde x \in p^{-1}(x)$, $\tilde x\in C(\tilde f, \tilde g)$. If $x$ were to belong to an inessential class of $f,g$, it is easy to see that $\tilde x$ would also have to belong to an inessential coincidence class of $\tilde f, \tilde g$, a contradiction since there are no such classes. We thus conclude that $N(f,g)=R(f,g)$ when $N(f,g)>0$. 
\end{proof}

\section{Jiang type property}

In classical fixed point theory, selfmaps on Jiang spaces possess the following property: if $L(f)=0$ then $N(f)=0$; if $L(f)\ne 0$ then $N(f)=R(f)$. For coincidences, if $M$ and $N$ are closed orientable manifolds of the same dimension and $N$ is either a nilmanifold \cite{GW3} or a coset space of compact connected Lie group \cite{VW} then either $L(f,g)=0 \Rightarrow N(f,g)=0$ or $L(f,g)\ne 0 \Rightarrow N(f,g)=R(f,g)$. If $\dim M>\dim N$ and if both $M$ and $N$ are compact nilmanifolds then we have $N(f,g)=0 \Leftrightarrow R(f,g)=\infty$ and $N(f,g)=R(f,g)$ if $N(f,g)>0$ \cite{GW2}.

Let $M,N$ be two connected compact manifolds. We say that the pair $(M,N)$ has {\it the Jiang type property for coincidences} if for any maps $f,g:M\to N$, $N(f,g)=0 \Leftrightarrow R(f,g)=\infty$ and $N(f,g)=R(f,g)$ if $N(f,g)>0$.

By Theorem \ref{equivalent-conditions}, we see that any pair $(M,N)$ of compact nilmanifolds has the Jiang type property for coincidences (even if $\dim M< \dim N$).

The proof of Corollary \ref{infra} in fact extends to the following general fact.

\begin{proposition}
Suppose $M$ has a regular finite cover $\hat M$ such that $(\hat M, N)$ has the Jiang type property for coincidences. Then the pair $(M,N)$ also has the Jiang type property for coincidences.
\end{proposition}

In the sequel, we investigate the Jiang type property for coincidences. In particular, we explore conditions so that the pair $(M,N)$ possesses such property where $N$ is a compact infra-nilmanifold.


\end{document}